\documentclass[sn-mathphys-num]{sn-jnl}

\usepackage{graphicx}%
\usepackage{multirow}%
\usepackage{amsmath,amssymb,amsfonts}%
\usepackage{amsthm}%
\usepackage{mathrsfs}%
\usepackage[title]{appendix}%
\usepackage{xcolor}%
\usepackage{textcomp}%
\usepackage{manyfoot}%
\usepackage{booktabs}%
\usepackage{algorithm}%
\usepackage{algorithmicx}%
\usepackage{algpseudocode}%
\usepackage{listings}%

\usepackage{enumerate}%
\usepackage{lmodern}

\usepackage{accents}

\newtheorem{theorem}{Theorem}
\newtheorem{remark}{Remark}%
\newtheorem{lemma}{Lemma}

\newtheorem{definition}{Definition}%

\begin{document}

\title[Control of a Lotka-Volterra System]{Control of a Lotka-Volterra System with Weak Competition}


\author*[1,2]{\fnm{Maicon} \sur{Sonego}}\email{mcn.sonego@unifei.edu.br}
\equalcont{These authors contributed equally to this work.}
\author[2,3,4]{\fnm{Enrique} \sur{Zuazua}}\email{enrique.zuazua@fau.de}
\equalcont{These authors contributed equally to this work.}


\affil*[1]{\orgdiv{Instituto de Matemática e Computação}, \orgname{Universidade Federal de Itajubá}, \orgaddress{\city{Itajubá}, \postcode{37500-903}, \state{Minas Gerais}, \country{Brazil}}}

\affil[2]{\orgdiv{Chair for Dynamics, Control, Machine Learning, and Numerics, Alexander von Humboldt-Professorship, Department of Mathematics}, \orgname{Friedrich-Alexander-Universität Erlangen-Nürnberg}, \orgaddress{ \city{Erlangen}, \postcode{91058}, \country{Germany}}}

\affil[3]{\orgdiv{Departamento de Matemáticas}, \orgname{Universidad Autónoma de Madrid}, \orgaddress{ \city{Madrid}, \postcode{28049}, \country{Spain}}}

\affil[4]{\orgdiv{Chair of Computational Mathematics}, \orgname{Fundación Deusto}, \orgaddress{ \city{Bilbao}, \postcode{48007}, \country{Basque Country}}}

\abstract{The Lotka-Volterra model reflects real ecological interactions where species compete for limited resources, potentially leading to coexistence, dominance of one species, or extinction of another. Comprehending the mechanisms governing these systems can yield critical insights for developing strategies in ecological management and biodiversity conservation.   
    
    In this work, we investigate the controllability of a Lotka-Volterra system modeling weak competition between two species. Through constrained controls acting on the boundary of the domain, we establish conditions under which the system can be steered towards various target states. More precisely, we show that the system is controllable in finite time towards a state of coexistence whenever it exists, and asymptotically controllable towards single-species states or total extinction, depending on domain size, diffusion, and competition rates. Additionally, we determine  scenarios where controllability is not possible and, in these cases, we construct barrier solutions that prevent the system from reaching specific targets. Our results offer critical insights into how competition, diffusion, and spatial domain influence species dynamics under constrained controls. Several numerical experiments complement our analysis, confirming the theoretical findings.}

\keywords{Lotka-Volterra, Boundary Control, Weak Competition, Barrier Solutions}


\pacs[MSC Classification]{92B05, 35Q93, 93-10}

\maketitle

\section{Introduction}
\subsection{Motivations}
In natural ecosystems, populations coexist within shared habitats, engaging in various ecological interactions such as interspecific competition, predation, parasitism, and mutualism. 
The interactions between two species in a region (such as predator-prey dynamics or competition) are described by the classical Lotka-Volterra models, which aim to predict various scenarios, such as species coexistence or the extinction of one or both species, through the analysis of the model coefficients. These coefficients are fundamentally related to the diffusion capacity and the intrinsic and inter-specific interactions between the species. The literature on this topic is extensive and diverse; for example, we cite \cite{PB, LE, LE2, MAY, MSM, tra2}.

An ever-present challenge, with a significant practical impact, is the ability to control specific phenomena, whether they are biological, chemical, economic, or social. In other words, it involves developing methods to steer natural or engineered phenomena towards a specific outcome. Broadly speaking, this is the essence of Control Theory, which aims to use mathematical analysis to identify conditions and methods that make this possible. The basic idea is to introduce a control agent into the model and then determine whether this agent can guide the model's solutions towards the desired objective over time.

For a long time, attempts have been made to control biological models, especially those involving the interaction of species in the same environment. For example, in \cite{PP}, the authors studied control in a Lotka-Volterra prey-predator system. The model simulates the interaction of a predator and a pest (the prey), and two types of controls stand out: an insecticide that kills only the predators and also the release rate of prey  raised in the laboratory. The controlled systems were considered without the effect of spatial diffusion:
$$\left\{\begin{array}{l}
\dfrac{dN_1}{dt}=(\alpha_1-\beta_1N_2)N_1\\\\
\dfrac{dN_2}{dt}=(\beta_2N_1-\alpha_2)N_2-U(t)N_2\\\\
N_1(t_0)=N_{10},\ \ N_2(t_0)=N_{20}
\end{array}\right.$$
and
$$\left\{\begin{array}{l}
\dfrac{dN_1}{dt}=(\alpha_1-\beta_1N_2)N_1+V(t)\\\\
\dfrac{dN_2}{dt}=(\beta_2N_1-\alpha_2)N_2\\\\
N_1(t_0)=\tilde{N}_{10},\ \ N_2(t_0)=\tilde{N}_{20}
\end{array}\right.$$
where \(N_1(t)\) and \(N_2(t)\) are the quantities of prey and predators, respectively. The parameters \(\alpha_i, \beta_i\) are positive constants, \(U(t)\) is the application rate of an insecticide that kills only the predators, and $V(t)$ is the release rate of the pests bred in the laboratory.

The results show that it is feasible to control a pest using an insecticide that kills only the predators and, additionally, it is feasible to manage a pest by releasing it at the appropriate time and rate. The biological rationale behind these two ideas is that these control variables can be used to prevent a subsequent collapse of the predator population, resulting in a pest resurgence.  Some practical examples of this method of biological control are considered in \cite{P1}  in the control of the red spider mite and in \cite{P2} where the release of hosts has also been used in the control of the cabbage worm in a host-parasite system. More recent research involving other species interaction models, which also served as motivation for the present work, can be found in \cite{BJ,oc1,oc2} and references therein.

Our objective here is to consider a Lotka-Volterra model that describes the weak competition between two species in a confined region. We propose control parameters that represent the population densities of the species on the boundary of this region. Our interest lies in determining whether it is possible to steer the model's solutions towards permanent coexistence scenarios or the extinction of one or both species.

An example of weak competition between two species in a confined region can be illustrated by considering two species, {\bf\it A} and {\bf\it  B}, that share overlapping but not identical niches. Species {\bf\it A} primarily consumes grass, while species {\bf\it  B} feeds mostly on shrubs, resulting in limited direct competition for resources. The confined region contains a balanced mix of both grass and shrubs, ensuring that neither species depletes the primary food source of the other. A fundamental aspect of this dynamic is what occurs on the boundary of the region. Although the weak competition in the example leads us to always expect a state of coexistence between species {\bf\it A} and {\bf\it  B}, our results show that it is possible to control the population density of species {\bf\it A} and {\bf\it  B} on the boundary of the region, in a way that guides the species over time, not only towards coexistence but also towards the extinction of one or even both species.

We can interpret that, even though they do not primarily compete for the same food source, a controlled proportion of populations {\bf\it A} and {\bf\it  B} on the boundary of the region can affect the diffusion, reproduction, and survival capabilities of one of the species, leading to its extinction over time. Furthermore, as one might naturally expect, our results show that the smaller the confinement region, the greater the effect of boundary control on the species.

The importance of such control capability is evident; just as we might desire the coexistence of species in a given environment or maintain the coexistence of two vital bacteria within an organism, we might also aim for the elimination of a pest from a crop. From a mathematical perspective, the challenges are also significant. The biological origins of the model inherently compel us to impose constraints on both the solutions and the controls, thereby making the application of many established techniques and results more difficult or even impossible. Below, we detail the problem under consideration and provide a brief overview of the relevant literature.

\subsection{Problem Setting} In this work, we consider a reaction-diffusion system in a one-dimensional domain primarily to enhance clarity in the exposition. Indeed, all results are valid in domains of higher dimensions, and Section \ref{HD} is dedicated to this case. Our problem is described below
\begin{equation}\label{SP}
\left\{\begin{array}{ll}
u_t=d_1u_{xx}+u(1-u-k_1v),& (x,t)\in (0,L)\times\mathbb{R}^+\\
v_t=d_2v_{xx}+v(a-v-k_2u),& (x,t)\in (0,L)\times\mathbb{R}^+\\
u(x,0)=u_0(x),\ \ v(x,0)=v_0(x),& x\in (0,L)
\end{array}\right.
\end{equation}
where
\begin{itemize}
\item $u$ and $v$ are the population densities of the two species competing  in $(0,L)$ and $(u(x,t),v(x,t))$ is the state to be controlled;
\item $u_0\in L^{\infty}((0,L);[0,1])$ and $v_0\in L^{\infty}((0,L);[0,a])$ are the initial conditions;
    \item $d_1,d_2>0$ are constants representing the diffusion rates;
    \item $a>0$ is a constant representing the intrinsic growth rate of $v$;
    \item $k_1,k_2>0$ are constants representing the  inter-specific competition between $u$ and $v$.
\end{itemize}

The intrinsic growth rate of \( u \), as well as the intraspecific competition rates, were set to $1$ in order to reduce the number of free variables in the model and thus facilitate its analysis. On the other hand, \eqref{SP} can be regarded as a suitably renormalized version of a  general  Lotka-Volterra model (see \cite{YY}, for instance).  

Note that $1$ and $a$ are the carrying capacities of $u$ and $v$, respectively, and therefore it is natural to constrain the solutions by these values, i.e.
\begin{equation}\label{VU}0\leq u(x,t)\leq 1\mbox{ and } 0\leq v(x,t)\leq a\mbox{ for all } (x,t)\in (0,L)\times\mathbb{R}^+.\end{equation}

Moreover, we assume
$$k_1,k_2<1$$
and this condition results in a {\it weak competition system}.

We suppose {\it boundary controls constraints} $c_u(x,t), c_v(x,t)\in L^{\infty}(\{0,L\}\times\mathbb{R}^+)$,
\begin{equation}
\left\{\begin{array}{ll}
u(x,t)=c_u(x,t)& (x,t)\in\{0,L\}\times\mathbb{R}^+\\
v(x,t)=c_v(x,t)& (x,t)\in\{0,L\}\times\mathbb{R}^+\\

\end{array}\right.
\end{equation}
satisfying \begin{equation}\label{res}
0\leq c_u\leq 1\mbox{ and }0\leq c_v\leq a.\end{equation}

Again, the constraints on the controls are natural due to the carrying capacities of $u$ and $v$.

Classical comparison results for parabolic systems, to be discussed later (Section \ref{PR}), ensure that with the boundary control constraints \eqref{res} and initial conditions $(u_0(x),v_0(x))$ satisfying $0\leq u_0\leq 1$, $0\leq v_0\leq a$, problem \eqref{SP} always has a unique solution $(u,v)$ satisfying \eqref{VU}. This follows from the monotonicity structure of \eqref{SP} and the fact that it is satisfied by both $(1,0)$ and $(0,a)$.

\begin{definition}\label{defp}
    We say that \eqref{SP} is {\it controllable towards} $(\overline{u},\overline{v})$ if  for any initial condition $(u_0(x),v_0(x))$ ($0\leq u_0\leq 1$, $0\leq v_0\leq a$),
there exist controls $c_u\in L^{\infty}(\{0,L\}\times\mathbb{R}^+; [0, 1])$, $c_v\in L^{\infty}(\{0,L\}\times\mathbb{R}^+; [0, a])$ such that
$$(u(t,x),v(t,x))\to (\overline{u},\overline{v})$$
uniformly in $[0,L]$ as $t\to\infty$.
\end{definition}

\begin{definition} We say that \eqref{SP} is {\it controllable in finite time towards} $(\overline{u},\overline{v})$ if  for any initial condition $(u_0(x),v_0(x))$ ($0\leq u_0\leq 1$, $0\leq v_0\leq a$),
there exist $T>0$ and controls $c_u\in L^{\infty}(\{0,L\}\times\mathbb{R}^+; [0, 1])$, $c_v\in L^{\infty}(\{0,L\}\times\mathbb{R}^+; [0, a])$ such that $(u(x,T),v(x,T))=(\overline{u}(x),\overline{v}(x))$.
\end{definition}

We are interested in controlling system \eqref{SP} towards the  steady-states of \eqref{SP}, that is, solutions of
\begin{equation}\label{SPSS}
\left\{\begin{array}{ll}
d_1u_{xx}+u(1-u-k_1v)=0,& x\in (0,L)\\
d_2v_{xx}+v(a-v-k_2u)=0,& x\in (0,L).\\
\end{array}\right.
\end{equation}

The targets to be considered in this work are described below.

\begin{itemize}
\item a homogeneous state of species coexistence
\begin{equation}\label{DH}(u^*,v^*)=\left(\dfrac{1-k_1a}{1-k_1k_2},\dfrac{a-k_2}{1-k_1k_2}\right),\end{equation}
which only makes sense to us when \(k_2<a<1/k_1\);
\item the extinction of the species $(0,0)$;
\item the survival of one of the species $(1,0)$ and $(0,a)$;
\item a heterogeneous state of species coexistence  for the case $d_1=d_2=d$ and $a=1$,
\begin{equation}\label{uev}
(u^{**}(x),v^{**}(x))=\left(\left(\dfrac{1-k_1}{1-k_1k_2}\right)\theta(x),\left(\dfrac{1-k_2}{1-k_1k_2}\right)\theta(x)\right)
\end{equation}
where $\theta(x)$ is a smooth function that satisfies
\begin{equation}\label{ET}
\left\{\begin{array}{l}
d\theta''(x)+\theta(x)(1- \theta(x))=0\ \  x\in (0,L) \\
\theta(0)=\theta(L)=0,\\
0<\theta(x)<1,\ \  x\in (0,L).
\end{array}\right.
\end{equation}

\end{itemize}

It is not difficult to see that, under our hypotheses, both steady states $(u^*, v^*)$ and $(u^{**}, v^{**})$ satisfy the conditions imposed in \eqref{VU}.

\begin{remark}
    Initially, we are interested in asymptotic controllability results (Definition \ref{defp}) which, for simplicity, will be referred to as controllability throughout the text. Controllability in finite time will only be possible for the target $(u^*,v^*)$, as it lies far from the limits imposed by the constraints in \eqref{VU} (Theorem \ref{T1F}). For this purpose, our result on asymptotic controllability towards $(u^*,v^*)$ is essential. Indeed, once the trajectory is sufficiently close to target $(u^*,v^*)$  (Theorem \ref{T1}), local controllability results can be employed to reach the target in finite time. This is possible because, locally, the control can oscillate above and below the target, allowing the trajectory to reach it in finite time. For this reason, controllability in finite time is not expected towards the targets $(1,0)$, $(0,a)$, $(0,0)$ and $(u^{**},v^{**})$, as in these cases, the local oscillation of controls is inhibited by the constraint \eqref{res}. A detailed discussion of this case is provided in Section \ref{FT}.

\end{remark}

Our results relate the system coefficients ($d_1,d_2,a,k_1,k_2$) and the interval length (\(L\)) to determine whether controllability towards the targets mentioned above is possible.


There is an extensive and diverse literature on Lotka-Volterra models, as well as on control problems in parabolic systems.
Specifically, for a Lotka-Volterra system with two species, results on null controllability were obtained in \cite{null} (refer to the references therein) using \(L^2\) controls acting within the interior of the domain. In this same context, results on optimal control for models without diffusion can be found in \cite{oc1,oc2,oc3,PP} and references therein.

In short, we can say that the main novelties of the present text are twofold: the controls act on the boundary of the domain, and they are subject to natural constraints. To the best of our knowledge, this is the first work to study the controllability in  this case for a coupled parabolic system. Such conditions were recently addressed for reaction-diffusion equations using different techniques \cite{Z1,Z2,Z3,MR}, such as phase portrait analysis and the staircase method; however, these cannot be directly applied in our case as we deal with a system rather than scalar equations.




It is important to note that, although we focus on Dirichlet controls in this work, the same results can be extended to Neumann controls. Specifically, once we identify Dirichlet controls that steer the solution toward the desired target, the resulting flux through the boundary can naturally be interpreted as a Neumann control. This duality arises in problems where controls act on the entire boundary. Consequently, the results obtained here can also be understood as strategies for controlling the boundary flux, which has practical applications in processes such as boundary regulation, filtration, or the facilitation of migration.

Our work is organized as follows. In Section \ref{MR}, we present the main results of this work, followed by all the proofs in Section \ref{PR}. In Section \ref{CS}, we introduce some different control strategies, while in Section \ref{FT}  we study the case of controllability in finite time.
 In the Section \ref{HD}, we extend the discussion to the same problem in a higher dimension. Numerical simulations are provided in Section \ref{NS}, and in Section \ref{CR}, we conclude with some final remarks and highlight several open problems for future research.


\section{Statement of the main results}\label{MR}

All the controllability and non-controllability results given below have an interesting biological interpretation. In summary, in this weak competition model, the ability to control the system towards a given target is closely linked to the existence of a homogeneous coexistence state and the size of the environment that the species inhabit. A more detailed explanation is provided after the statement of each theorem.

\begin{theorem}\label{T1}
If \begin{equation}\label{CC}k_2<a<\dfrac{1}{k_1}\end{equation}
then \eqref{SP} is controllable towards $(u^*,v^*)$ defined in \eqref{DH}.
\end{theorem}

Regarding the results above, several observations can be made.
Condition \eqref{CC} is necessary for the existence of a coexistence state \((u^*, v^*)\) in the weak competition regime we are assuming; in other words, its assumption is crucial in the study of controllability towards the target \((u^*, v^*)\). What the theorem tells us is that \eqref{CC} is also sufficient for controllability.

Controllability towards the target \((u^*, v^*)\),  independent of the domain size \(L\) or other parameter values, highlights the decisive role of species competition in relation to diffusion capacities and the environment's size. Even if one of the species has a higher diffusion capacity in a large competition region, given that this competition is weak, it is possible to control the boundary of this environment in order to achieve a coexistence state of the species.

The proof of Theorem \ref{T1} consists of presenting a control strategy that asymptotically drives a given initial state \((u_0,v_0)\) to the target \((u^*, v^*)\). As we will see, the strategy presented is derived from the dynamics of the same problem but with Neumann boundary conditions. A control strategy based on traveling waves is introduced when \(L\) and the reproduction rate of $v$ are sufficiently small (see Subsection  \ref{TW}).

\begin{theorem}\label{T2}

\begin{enumerate}[$(i)$]
\item If $L\leq \displaystyle\sqrt{\frac{d_2}{a}}\pi$ or
 $k_2>a$ then \eqref{SP} is controllable towards $(1,0)$.
 \item If $L\leq \sqrt{d_1}\pi$ or
 $k_1>\dfrac{1}{a}$ then \eqref{SP} is controllable towards $(0,a)$.
 \end{enumerate}
\end{theorem}

This theorem illustrates the complex balance between diffusion, competition, and domain size in determining species dominance. Specifically, the conditions indicate how the length of \(L\), the diffusion rate, and the competition intensity influence the possibility to steer the system towards  states \((1,0)\) or \((0,a)\). In both cases ($(i)$ and $(ii)$), for fixed diffusion, reproduction, and competition parameters, it is always possible to control the system towards these single-species state if \(L\) is sufficiently small. In these cases, static controls are sufficient.

On the other hand, if the competition and reproduction rates are related such that \(k_2 > a\) or \(k_1 > 1/a\) (and consequently we do not have a coexistence state), then it is possible to control the system towards \((1,0)\) or \((0,a)\), respectively, regardless of the size of \(L\).

At this point, a natural question is to inquire about the controllability towards the states $(1,0)$ or $(0,a)$ when none of these conditions are satisfied. The answer is provided in the following theorem.

\begin{theorem}\label{T4}
If $k_2<a<1/k_1$ and
\begin{equation}\label{LM}L>\max\left\{\sqrt{\dfrac{d_1}{1-ak_1}}\pi,\sqrt{\dfrac{d_2}{a-k_2}}\pi\right\}\end{equation}
then \eqref{SP} is not controllable towards either $(1,0)$ or $(0,a)$. 
\end{theorem}

Note that \eqref{LM} contradicts all conditions present in Theorem \ref{T2}. That is, in the above theorem, besides having $k_2<a<1/k_1$, we are assuming $L$ to be sufficiently large. Non-controllability is demonstrated through the construction of barrier solutions that prevent certain initial states from approaching  
 $(1,0)$ or $(0,a)$.

The above result can be interpreted as follows: if $k_2<a<1/k_1$ and the environment is sufficiently large (the threshold is explicitly given in \eqref{LM}), for certain initial states of the species, it is impossible to intervene in the population quantities on the boundary of the environment in such a way that one of them goes extinct over time.

However, for a sufficiently small environment, controllability towards the extinction of both species is always possible.

\begin{theorem}\label{T3}

  \begin{enumerate}[$(i)$]
  
  \item If  \begin{equation}\label{LT3} L\leq \min\left\{\sqrt{d_1}\pi,\sqrt{\frac{d_2}{a}}\pi\right\}\end{equation}
  then the system \eqref{SP} is controllable towards $(0,0)$;
  \item if 

  $$a<\dfrac{1}{k_1}\mbox{ and } L>\sqrt{\dfrac{d_1}{1-ak_1}}\pi$$
  or
  $$a>k_2\mbox{ and }  L>\sqrt{\dfrac{d_2}{a-k_2}}\pi,$$
  then the system \eqref{SP} is not controllable towards $(0,0)$.
\end{enumerate}
\end{theorem}

In a regime of weak competition, with fixed diffusion and reproduction coefficients, controllability towards species extinction is exclusively determined by the length of \(L\). If \(L\) satisfies \eqref{LT3}, then any initial state is driven to extinction with static controls equal to zero on the boundary.

Again, if $k_2<a<1/k_1$ and $L$ is sufficiently large, it is possible to construct a barrier solution that prevents certain initial states from approaching extinction.

Finally, assuming equal diffusion coefficients (\(d_1 = d_2 = d\)) and the same reproduction rates (\(a = 1\)), it is possible to steer the system to a heterogeneous equilibrium state of coexistence, provided that the length of \(L\) is sufficiently large.

\begin{theorem}\label{T5}

  If $a=1$, $d_1=d_2=d>0$ and $ L> \sqrt{d}\pi$,
  then  \eqref{SP} is  controllable towards  an specific heterogeneous coexistence state $(u^{**},v^{**})$ defined in \eqref{uev}.

\end{theorem}

In the theorem above, controllability is again achieved with static controls equal to zero. We have established the existence and uniqueness of a positive equilibrium solution \((u^{**},v^{**})\) to problem \eqref{SP}  with zero Dirichlet boundary conditions. In this case, such a solution will attract every trajectory with initial state \((u_0,v_0)\) ($0\leq u_0\leq 1$, $0\leq v_0\leq a$).

A simple biological interpretation of the result above is that for species with the same reproduction characteristics and diffusion capacity, if the interaction between them has a small effect, and if the length of \(L\) is sufficiently large, then a coexistence state assuming values zero at the boundary will exist and be unique and hence stable. In this controlled situation, neither species can drive the other to extinction regardless of the initial state taken.


Finally, it is crucial to emphasize the significance of non-controllability results. In the present context, broadly speaking, a non-controllability result towards a target  $(\overline{u},\overline{v})$  indicates that at least one initial state cannot be steered to  
$(\overline{u},\overline{v})$, regardless of the boundary conditions. That is, there are no controls acting on the boundary that can bring that initial state closer to the target $(\overline{u},\overline{v})$. These results are obtained by constructing barrier solutions that prevent a certain initial condition from approaching the target.


\begin{figure}[ht!]
  \centering
   \includegraphics[width=0.6\linewidth]{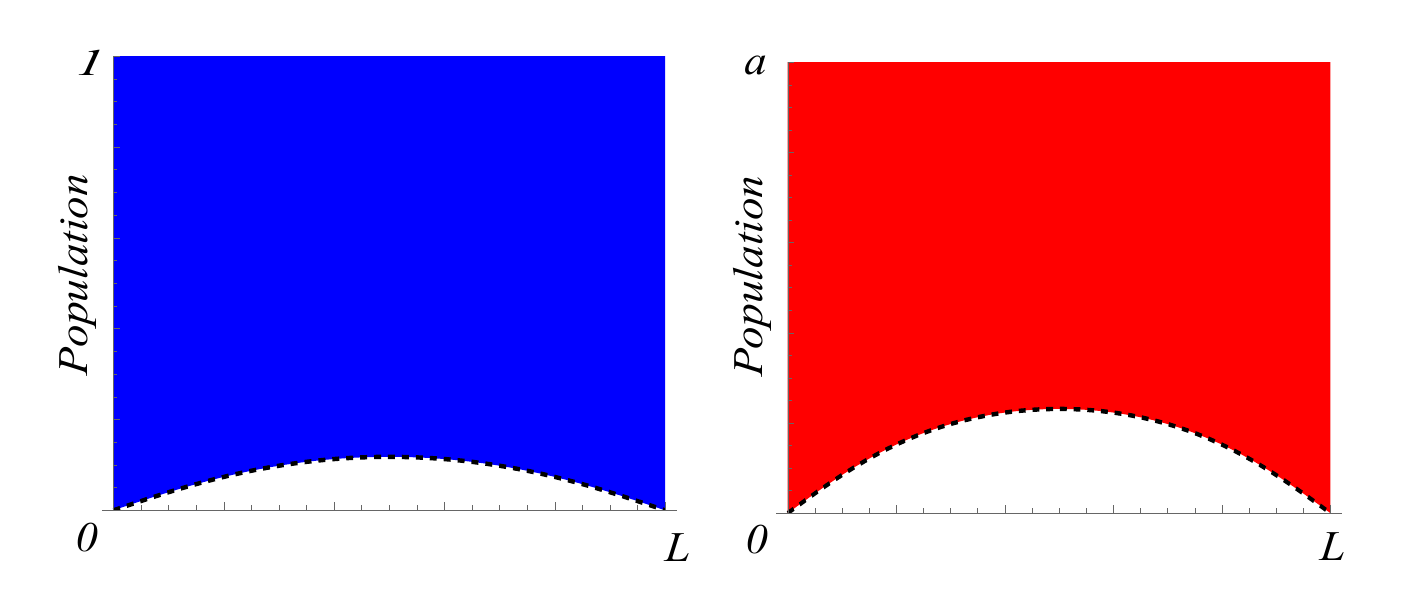}
 \caption{Barrier functions that prevent $u$ and/or $v$ from extinction.}
   \label{ilu}
\end{figure}

Figure \ref{ilu} illustrates this interpretation when there is no controllability towards $(0,0)$. In the graph on the left, we have the barrier solution (dashed black line) preventing the population $u$ from approaching $0$. That is, if $(u_0, v_0)$ is an initial datum such that $u_0$ lies in the blue region of the figure, then for any controls satisfying \eqref{res}, the corresponding solution $u$ will remain in the blue region. On the other hand, $v$ may approach $0$, unless $v_0$ is also above its respective barrier solution, that is, in the red region in the graph on the right.
More details about this interpretation can be seen in the following sections.

\section{Proofs of the main results}\label{PR}

\subsection{Preliminaries}


The dynamics of problem \eqref{SP} under Neumann boundary conditions will be of fundamental importance in the control strategies considered here. Below, we state a well-known result adapted for our problem. As a reference, we cite \cite{YY,lida}.

\begin{theorem}\label{TN}
Consider \eqref{SP} under Neumann boundary conditions, that is
\begin{equation}\label{NC}
u_x(0,t)=u_x(L,t)=0,\ \ v_x(0,t)=v_x(L,t)=0\ \ t\in\mathbb{R}^+.
\end{equation}
If $(\overline{u}(x,t),\overline{v}(x,t))$ is a solution of \eqref{SP}, \eqref{NC} with initial conditions $0\leq u_0\leq 1$, $0\leq v_0\leq a$,  then $0\leq \overline{u}(x,t)\leq 1 $, $0\leq \overline{v}(x,t)\leq a$ for $(x,t)\in [0,L]\times\mathbb{R}^+$ and we have the following asymptotic behaviors of $(\overline{u},\overline{v})$:
\begin{enumerate}[$(i)$]
\item if $k_2<a<\dfrac{1}{k_1}$,
$$\displaystyle\lim_{t\to\infty}(\overline{u}(x,t),\overline{v}(x,t))=(u^*,v^*)$$
uniformly in $[0,L]$;
\item if $k_2>a$,
$$\displaystyle\lim_{t\to\infty}(\overline{u}(x,t),\overline{v}(x,t))=(1,0)$$
uniformly in $[0,L]$;
\item if $a>\dfrac{1}{k_1}$,
$$\displaystyle\lim_{t\to\infty}(\overline{u}(x,t),\overline{v}(x,t))=(0,a)$$
uniformly in $[0,L]$.

\end{enumerate}
\end{theorem}

The  existence-comparison theorem below is an adaptation of Theorem 2.3 in \cite{pao}.

\begin{theorem}\label{cos}
Let $(\tilde{u},\tilde{v})$, $(\undertilde{u},\undertilde{v})$ be a pair of smooth functions  such that $\tilde{u}\geq\undertilde{u}\geq 0$ and $\tilde{v}\geq\undertilde{v}\geq 0$. Moreover, suppose that $(\tilde{u},\undertilde{v})$ satisfies 
\begin{equation}\label{SPEX}
\left\{\begin{array}{ll}
\tilde{u}_t\geq d_1\tilde{u}_{xx}+\tilde{u}(1-\tilde{u}-k_1\undertilde{v}),& (x,t)\in (0,L)\times\mathbb{R}^+\\
\undertilde{v}_t\leq d_2\undertilde{v}_{xx}+\undertilde{v}(a-\undertilde{v}-k_2\tilde{u}),& (x,t)\in (0,L)\times\mathbb{R}^+\\
\tilde{u}(x,0)\geq u_0(x),\ \ \undertilde{v}(x,0)\leq v_0(x),& x\in (0,L)\\
\tilde{u}(x,t)\geq 0, \ \ \undertilde{v}(x,t)\leq 0,& (x,t)\in\{0,L\}\times\mathbb{R}^+,
\end{array}\right.
\end{equation}
and that $(\undertilde{u},\tilde{v})$ satisfies
the corresponding reversed inequalities. Then the problem \eqref{SP} under zero Dirichlet boundary conditions has a
unique solution $(u, v)$ such that
$$\undertilde{u}(x,t)\leq u(x,t)\leq \tilde{u}(x,t),\ \ \undertilde{v}(x,t)\leq v(x,t)\leq\tilde{v}(x,t),$$
for $(x,t)\in [0,L]\times\mathbb{R}^+$.
\end{theorem}

The theorem below is an adaptation of Theorem 4.1 in \cite{ZPAO}.

\begin{theorem}\label{cpao}
Let $(\tilde{u},\tilde{v})$, $(\undertilde{u},\undertilde{v})$ be a pair of smooth functions  and suppose that
\begin{equation}\label{SPCOM}
\left\{\begin{array}{ll}
\tilde{u}_t\geq d_1\tilde{u}_{xx}+\tilde{u}(1-\tilde{u}-k_1\undertilde{v}),& (x,t)\in (0,L)\times\mathbb{R}^+\\
\undertilde{v}_t\leq d_2\undertilde{v}_{xx}+\undertilde{v}(a-\undertilde{v}-k_2\tilde{u}),& (x,t)\in (0,L)\times\mathbb{R}^+
\end{array}\right.
\end{equation}
and that $(\undertilde{u},\tilde{v})$ satisfies
the corresponding reversed inequalities. Moreover, suppose that
\begin{equation}\label{SPCOM2}
\left\{\begin{array}{ll}
\tilde{u}(x,0)\geq \undertilde{u}(x,0),\ \ \undertilde{v}(x,0)\leq \tilde{v}(x,0),& x\in (0,L)\\
\tilde{u}(x,t)\geq  \undertilde{u}(x,t), \ \ \undertilde{v}(x,t)\leq \tilde{v}(x,t),& (x,t)\in\{0,L\}\times\mathbb{R}^+,
\end{array}\right.\end{equation}
then
$$\undertilde{u}(x,t)\leq  \tilde{u}(x,t),\ \ \undertilde{v}(x,t)\leq \tilde{v}(x,t),$$
for $(x,t)\in [0,L]\times\mathbb{R}^+$.

\end{theorem}

For the upcoming theorems, the smallest eigenvalue $\lambda_0$ and its corresponding eigenfunction $\phi$ of the problem below will be important to us.
\begin{equation}\label{AS}
\left\{\begin{array}{ll}
\phi''(x)+\lambda\phi(x)=0,& x\in (0,L)\\
\phi(0)=\phi(L)=0.
\end{array}\right.
\end{equation}

It is easy to see that 
\begin{equation}\label{EL}
\lambda_0=\dfrac{\pi^2}{L^2}
\mbox{ and }
\phi(x)=\sin\left(\frac{\pi}{L}x\right).
\end{equation}

Finally, we state the well-known result regarding the existence of positive solution for the logistic equation with diffusion (see \cite{PAY}).

\begin{lemma}\label{lemp}If $\alpha>\lambda_0$ and $\beta>0$, then  there exists a unique smooth function $z$ such that
\begin{equation}\label{EZP}
\left\{\begin{array}{l}
z''(x)+z(x)(\alpha-\beta z(x))=0\ \  x\in (0,L) \\
z(0)=z(L)=0,\\
z(x)>0,\ \  x\in (0,L).
\end{array}\right.
\end{equation}
\end{lemma}

\subsection{Proofs of theorems \ref{T1} -- \ref{T5}}

Here we present the proofs of our main theorems; for convenience, we have reversed the order of theorems \ref{T4} and \ref{T3}.

\begin{proof}[Proof of Theorem \ref{T1}]


Consider an initial condition $(u_0(x),v_0(x))$ such that $0\leq u_0(x)\leq 1$, $0\leq v_0(x)\leq a$,  for all $x\in (0,L)$.

The control strategy consists of considering 
$$(c_u(x,t),c_v(x,t))=(\overline{u}(x,t),\overline{v}(x,t))$$
for $(x,t)\in\{0,L\}\times\mathbb{R}^+$ where $(\overline{u}(x,t),\overline{v}(x,t))$ is the solution of \eqref{SP}, \eqref{NC} with the initial conditions given above. Note that by Comparison Principle, $0\leq\overline{u}\leq 1$ and $0\leq\overline{v}\leq a$. The result follows as a consequence of Theorem \ref{TN} $(i)$.
\end{proof}

\begin{proof}[Proof of Theorem \ref{T2}]
We will only demonstrate $(i)$ since the demonstration of $(ii)$ is analogous.

We begin by assuming that
\begin{equation}\label{L}
L\leq\sqrt{\dfrac{d_2}{a}}\pi
\end{equation}
and then we consider the following static controls
$$(c_u(x,t),c_v(x,t))\equiv(1,0)$$
for $(x,t)\in\{0,L\}\times\mathbb{R}^+$. We will prove that, in this case, $(u^s,v^s)\equiv (1,0)$ is the unique stationary state of \eqref{SP} under boundary conditions given by $(u(x,t),v(x,t))=(1,0)$ for $(x,t)\in \{0,L\}\times\mathbb{R}^+$.


Let $(u^s,v^s)$ be any nonnegative solution of 
\begin{equation}\label{SPS}
\left\{\begin{array}{ll}
d_1u_{xx}+u(1-u-k_1v)=0,& x\in (0,L)\\
d_2v_{xx}+v(a-v-k_2u)=0,& x\in (0,L)\\
u(0)=u(L)=1,\ \ v(0)=v(L)=0.& 
\end{array}\right.
\end{equation}

By multiplying the second equation in \eqref{SPS} by $\phi(x)=\sin(\pi x/L)$ (see \eqref{AS}), integrating over $(0,L)$ and using Green's theorem twice together with the boundary condition $v(0)=v(L)=0$, we obtain
$$-\displaystyle\int_0^Ld_2\phi_{xx}v^sdx=\int_0^L\phi v^s(a-v^s-k_2u^s)dx.$$
By \eqref{AS} and the non-negativity of $u^s$ and $v^s$,
$$\displaystyle\int_0^L\phi v^s\left(\dfrac{d_2\pi^2}{L^2}-a\right)dx\leq -\int_0^L\phi(v^s)^2dx.$$
As we are assuming $\left(\dfrac{d_2\pi^2}{L^2}-a\right)\geq 0$ (see \eqref{L}) we conclude that $v^s\equiv 0$. 

Now, we have the following problem
\begin{equation}\label{SPSu}
\left\{\begin{array}{ll}
d_1u_{xx}+u(1-u)=0& x\in (0,L)\\
u(0)=u(L)=1,&
\end{array}\right.
\end{equation}
and we state that $u^s\equiv 1$ is the only solution. Indeed, if we take $w=1-u^s$, then
\begin{equation}\label{SPSw}
\left\{\begin{array}{ll}
-d_1w_{xx}+w(1-w)=0& x\in (0,L)\\
w(0)=w(L)=0.&
\end{array}\right.
\end{equation}
By multiplying the equation by $w$ and integrating in $(0,L)$,
we get
\begin{equation}\label{SER}
\int_0^L d_1(w_x)^2dx=-\int_0^L w^2(1-w)dx.
\end{equation}
Since $w(0)=w(L)=0$,  \eqref{SER} is true only if $w\equiv 0$. It follows that $u^s\equiv 1$ and $(u^s,v^s)\equiv (1,0)$ is the unique stationary state of \eqref{SP} under boundary conditions given by $(u(x,t),v(x,t))=(1,0)$ for $(x,t)\in \{0,L\}\times\mathbb{R}^+$. By the main result of \cite{JS}, every bounded solution of \eqref{SP} converge to a steady-state. Thus, the controllability of the system is achieved.

If we have $k_2>a$ then  we proceed as in Theorem \ref{T1} but using $(ii)$ from Theorem \ref{TN}.
\end{proof}

\begin{proof}[Proof of Theorem \ref{T3}]
In order to prove $(i)$ we assume \eqref{LT3} and
we take the static controls $$(c_u(x,t),c_v(x,t))=(0,0)$$ for $(x,t)\in\{0,L\}\times\mathbb{R}^+$. Then we can proceed as in the first part of Theorem \ref{T2} to conclude that $(u^s,v^s)\equiv (0,0)$ is the unique stationary state of \eqref{SP} under boundary conditions given by $(u(x,t),v(x,t))=(0,0)$ for $(x,t)\in \{0,L\}\times\mathbb{R}^+$.

For $(ii)$, we prove that \eqref{SP} is not controllable towards $(0,0)$  due to the presence of a barrier function in these conditions. We suppose 
 $$a<\dfrac{1}{k_1}\mbox{ and } L>\sqrt{\dfrac{d_1}{1-ak_1}}\pi$$
  and the proof for the case
  $$a>k_2\mbox{ and }  L>\sqrt{\dfrac{d_2}{a-k_2}}\pi,$$
  is analogous.

Under these conditions, we have built a barrier that blocks  solutions $u$  from approaching $0$ for determined initial conditions. We denote by $\eta_1$ the unique smooth function that satisfies 
\begin{equation}\label{ET2}
\left\{\begin{array}{l}
d_1\eta''(x)+\eta(x)(1- \eta(x)-k_1a)=0,\ \  x\in (0,L) \\
\eta(0)=\eta(L)=0,\\
\eta(x)>0,\ \  x\in (0,L).
\end{array}\right.
\end{equation}
  The existence of $\eta_1$ is guaranteed by Lemma \ref{lemp}. For our convenience, we write
$$d_1\eta_1''(x)+f(\eta_1(x))=0$$
for $x\in (0,L)$ where $f(z)=z(1- z-k_1a)$. The conclusion follows by an application of the comparison principle to scalar equations (e.g., see \cite{PAO2}). First, we note that $\eta_1(x)\leq 1$ since $\tilde{\eta}(x)\equiv 1-k_1a\leq 1$ is a super-solution of \eqref{ET2}. Now, we take an initial condition $(u_0,v_0)$ such that $1\geq u_0(x)\geq \eta_1(x)$ for all $x\in (0,L)$. For any constrained  control $(c_u,c_v)$, the respective solution $(u,v)$, in particular, satisfies, 
\begin{equation}\label{SPZ}
\left\{\begin{array}{ll}
u_t=d_1u_{xx}+u(1-u-k_1v),& (x,t)\in (0,L)\times\mathbb{R}^+\\
u(x,t)=c_u(x,t),& (x,t)\in \{0,L\}\times\mathbb{R}^+\\
u(x,0)=u_0(x).
\end{array}\right.
\end{equation}

We note that
$$u_t=d_1u_{xx}+u(1-u-k_1v)\geq d_1u_{xx}+u(1-u-k_1a)=d_1u_{xx}+f(u).$$

It follows that $u(x,t) \geq \eta_1 (x)$ for all $(x,t)\in (0,L)\times\mathbb{R}^+$ and as $\eta_1(x) > 0$ for all $x \in (0,L)$ we have that it is impossible to drive $u$ from $u_0$ to $0$. Therefore, \eqref{SP} is not controllable towards $(0,0)$.
\end{proof}

If we assume $a>k_2$ and $  L>\sqrt{\dfrac{d_1}{a-k_2}}\pi$, then it is possible to construct a barrier solution for $v$ in an analogous manner, and again the system will not steer to (0,0) for certain initial conditions. Evidently, it may be of interest to inhibit both $u$ and $v$ from approaching $0$, and this will be possible if
$$L>\max\left\{\sqrt{\dfrac{d_1}{1-ak_1}}\pi,\sqrt{\dfrac{d_2}{a-k_2}}\pi\right\}$$
where, obviously, we are assuming $k_2<a<1/k_1$. This case will be further detailed in the next section.

\begin{proof}[Proof of Theorem \ref{T4}]
We will prove that \eqref{SP} is not controllable towards $(1,0)$. The proof to the target $(0,a)$ is analogous. Again, we show the existence of a barrier function, in this case, the existence of a nontrivial solution $(u^s,v^s)$ to
\begin{equation}\label{SPN}
\left\{\begin{array}{ll}
d_1u_{xx}+u(1-u-k_1v)=0,& x\in (0,L)\\
d_2v_{xx}+v(a-v-k_2u)=0,& x\in (0,L)\\
u(0)=u(L)=1,\ \ v(0)=v(L)=0.& 
\end{array}\right.
\end{equation}
Ir order to show the existence of such $(u^s, v^s)$, we will employ Theorem \ref{cos} for the elliptic case, that is, without terms involving derivatives with respect to $t$ and without requirements at $t=0$. Furthermore, we observe that Theorem \ref{cos} is stated with zero Dirichlet boundary conditions, however, the same holds for non-zero conditions.

Let $\tilde{u}\equiv 1$ and $\undertilde{v}=\delta\phi$ for some $\delta>0$. Then $(\tilde{u},\undertilde{v})$ satisfies the inequalities in \eqref{SPEX} if
\begin{equation}\label{e2}d_2\delta\phi_{xx}+\delta\phi(a-\delta\phi-k_2)\geq 0\end{equation}
since the inequality for \(\tilde{u}\) and the boundary inequalities for \(\tilde{u}\) and \(\undertilde{v}\) are trivially satisfied. Now, \eqref{e2} is equivalent to
$$d_2\dfrac{\pi^2}{L^2}\leq a-\delta\phi-k_2$$
or
$$\delta\phi\leq a -k_2-\dfrac{d_2}{L^2}\pi^2$$
and these inequalities occurs if we take $\delta=a -k_2-\dfrac{d_2}{L^2}\pi^2$ which is positive because \eqref{LM}.

Similarly, the pair $\tilde{v}\equiv a$ and $\undertilde{u}(x)=\eta\phi$ ($\eta>0$) satisfies the reversed inequality in
\eqref{SPEX}  if
$$d_1\eta\phi_{xx}+\eta\phi(1-\eta\phi-k_1a)\geq 0$$
and this is true if $\eta=1-k_ 1a-\dfrac{d_1}{L^2}\pi^2$ which is positive again because \eqref{LM}.

By Theorem \ref{cos}, there exists a solution $(u^s,v^s)$ of \eqref{SPN} such that
$$\eta\phi\leq u^s\leq 1\ \ \mbox{ and }\ \ \delta\phi\leq v^s\leq a.$$

Now, consider a initial condition $(u_0,v_0)$ such that $$0\leq u_0(x)\leq u^s(x) \mbox{ and } a\geq v_0(x)\geq v^s(x)$$ for all $x\in [0,L]$. Then for any controls $(c_u,c_v)$ satisfying \eqref{res}, we can use Theorem \ref{cpao} to conclude that the  solution $(u(x,t),v(x,t))$ of \eqref{SP} with these controls satisfies $u(x,t)\leq u^s(x)$ and $v(x,t)\geq v^s(x)$ for all $(x,t)\in [0,L]\times\mathbb{R}^+$. Note that $u^s\not\equiv 1$ because $v^s\geq \delta\phi$. It follows that, $u(x,t)$ and $v(x,t)$ does not converge to $1$ and $0$, respectively, and the theorem is proved.
\end{proof}

\begin{remark}
Note that Theorem \ref{T4} provides us with a general condition for the non-controllability of \eqref{SP} towards states in which one species goes extinct and the other survives at its maximum capacity. More specifically, this is achieved through the construction of a barrier solution for \( (1,0) \) and another barrier solution for \( (0,a) \). In a less restrictive way, we can relax the hypotheses and achieve, for example, a barrier solution that does not allow driving \( u \) from an initial condition \( u_0 \) to $0$ for any constrained control \( (c_u,c_v) \), but eventually allowing \( v \) to approach \( a \). This is sufficient for non-controllability towards \( (0,a) \). For this, it is enough to assume $$a<\dfrac{1}{k_1}\mbox{ and } L>\sqrt{\dfrac{d_1}{1-ak_1}}\pi$$ and proceed as in Theorem \ref{T3}.
\end{remark}

\begin{proof}[Proof of Theorem \ref{T5}]
The proof of this theorem is an application of Theorem 1.3 of \cite{COS} related to the stability of a stationary solution of \eqref{SP} with zero Dirichlet conditions. So, here, under our assumptions, we consider the controls $(c_u,c_v)\equiv (0,0)$ and it suffices to prove the existence and uniqueness of the coexistence state $(u^{**},v^{**})$ for \eqref{SP} under zero Dirichlet boundary conditions.

First, we prove existence using Lemma \ref{lemp}. 
We note that $\theta$ (see \eqref{uev}) is the unique smooth function that satisfies (recall that, in our case, $\lambda_0=\frac{\pi^2}{L^2}$ and  we are assuming $1/d>\pi^2/L^2$)
$$
\left\{\begin{array}{l}
d\theta''(x)+\theta(x)(1- \theta(x))=0\ \  x\in (0,L) \\
\theta(0)=\theta(L)=0,\\
\theta(x)>0,\ \  x\in (0,L).
\end{array}\right.
$$

A straightforward computation yields that $(u^{**}(x),v^{**}(x))$ given by \eqref{uev}
is a coexistence state for \eqref{SP} under zero Dirichlet boundary conditions.

In order to prove uniqueness, suppose that $(u,v)$ is a stationary solution to equation \eqref{SP} under Dirichlet boundary conditions zero such that $u(x)>0$ and $v(x)>0$ for $x\in (0,L)$.

{\it Claim:} if $z$ is a solution of
\begin{equation}\label{EZPU}
\left\{\begin{array}{l}
z''+z(1- u-v)=0\ \  x\in (0,L) \\
z(0)=z(L)=0,
\end{array}\right.
\end{equation}
then $z\equiv 0$.

Indeed, we note that for $\mu=0$, the problem below 
\begin{equation}\label{EZPW}
\left\{\begin{array}{l}
w''+w(1- u-k_1v)+\mu w=0\ \  x\in (0,L) \\
w(0)=w(L)=0,
\end{array}\right.
\end{equation}
has a positive solution $w=u$. Thus, if $\mu_1$ is the lowest eigenvalue of \eqref{EZPW}, then $\mu_1=0$. By  the variational characterization of $\mu_1$ in terms of Rayleigh quotients (see, for instance, \cite{HIL}), for any $\psi\in C^1(0,L)$, $\psi(0)=\psi(L)=0$, $\psi\not\equiv 0$, we have that
\begin{equation}\label{RAY}0\leq \dfrac{\int_0^L(\psi')^2-(1-u-k_1v)\psi^2dx}{\int_0^L\psi^2dx}.\end{equation}

Now, we multiply \eqref{EZPU} by $z$ and integrate on $(0,L)$. Using integration by parts and \eqref{RAY} we  conclude that
$$\begin{array}{lll}
0&=&\displaystyle\int_0^L(z')^2-z^2(1-u-v)dx\\
&=&\displaystyle\int_0^L(z')^2-z^2(1-u-k_1v)dx + \int_0^Lz^2v(1-k_1)dx\\
&\geq&\displaystyle\int_0^Lz^2v(1-k_1)dx
\end{array}$$
which implies that $z\equiv 0$ and the {\it Claim} is proved. We can say that the steady-states of \eqref{SP} satisfies (recall $a=1$, $d=d_1=d_2$)
\begin{equation}\label{SPP}
\left\{\begin{array}{ll}
du_{xx}+u(1-u-v)+(1-k_1)uv=0& (x,t)\in (0,L)\times\mathbb{R}^+\\
dv_{xx}+v(1-v-u)+(1-k_2)uv=0& (x,t)\in (0,L)\times\mathbb{R}^+.
\end{array}\right.
\end{equation}

Multiplying the first of these equations by $(1-k_2)$ and the second by $(1-k_1)$ and subtracting, we obtain
\begin{equation}\label{psi}
\left\{\begin{array}{ll}
\psi''+\psi(1-u-v),\ \  x\in (0,L)\\
\psi(0)=\psi(L)=0,
\end{array}\right.
\end{equation}
where $\psi=(1-k_2)u-(1-k_1)v$. The {\it Claim} proved above implies that $\psi\equiv 0$. Then $v=ru$ where $r=(1-k_2)/(1-k_1)$. Note that
$$u+k_1v=u+k_1ru=\dfrac{1-k_1k_2}{1-k_1}u.$$
It follows that
$$du''+u\left(1-\left(\dfrac{1-k_1k_2}{1-k_1}\right)u\right)=0.$$
Since $u(0)=u(L)=0$ and $u(x)>0$ for $x\in (0,L)$, by Lemma \ref{lemp} we have that $u(x)=\left(\dfrac{1-k_1}{1-k_1k_2}\right)\theta(x)$ and consequently $v(x)=\left(\dfrac{1-k_2}{1-k_1k_2}\right)\theta(x)$, i.e., $(u(x),v(x))=(u^{**}(x),v^{**}(x))$ for all $x\in [0,L]$ is the unique coexistence state for \eqref{SP} with control $(c_u,c_v)\equiv (0,0)$. As stated above, controllability is a consequence of the stability result that can be seen in  \cite[Theorem 3.1]{COS}.
\end{proof}


\section{Alternative control strategies}\label{CS}

\subsection{Travelling wave strategy towards the state of constant coexistence}\label{TW} For the particular case where, in addition to \eqref{CC}, we have  $a<1$ and $L\leq \sqrt{d_1}\pi$, we can establish a different control strategy in order to reach the target $(u^*,v^*)$. The idea is to use the traveling waves that exist in this case.

A travelling wave with velocity $c$ which connects the steady-states $(0,a)$ and $(u^*,v^*)$ of \eqref{SP}, is a pair
$$(\Psi,\Phi)(x,t):=(U,V)(\xi),\ \ \xi:=x+ct$$
such that $U,V:\mathbb{R}\to \mathbb{R}$ form a solution $(U,V)$ of
\begin{equation}\label{TW1}\left\{\begin{array}{l}
d_1U''(\xi)-cU'(\xi)+U(\xi)(1-U(\xi)-k_1V(\xi))=0\\
d_2V''(\xi)-cV'(\xi)+V(\xi)(a-V(\xi)-k_2U(\xi))=0
\end{array}\right.\end{equation}
for $\xi\in\mathbb{R}$ and
\begin{equation}\label{TW2}\left\{\begin{array}{l}
\displaystyle\lim_{\xi\to-\infty}(U,V)=(0,a)\\\\
\displaystyle\lim_{\xi\to\infty}(U,V)=(u^*,v^*)\\\\
U'(\xi)>0,\ \ V'(\xi)<0,\ \ \xi\in\mathbb{R}.
\end{array}\right.\end{equation}

With the hypotheses $a<1$ and \eqref{CC}, for each $c>2\sqrt{\dfrac{1-ak_1}{1-k_1k_2}}$, we use \cite{tra2} (see also \cite{tra}) to ensure the existence of a traveling wave of \eqref{SP} satisfying \eqref{TW1} and \eqref{TW2}. We recall that any translation of a traveling wave is still a traveling wave.

Let $(u_0,v_0)$ ($0\leq u_0\leq 1$, $0\leq v_0\leq a$) be an initial condition of \eqref{SP}. We start with the following controls:
\begin{equation}\label{CTW}
(\tilde{c_u}(x,t),\tilde{c_v}(x,t))=(0,a),\mbox{ for } (x,t)\in\{0,L\}\times (0,t_1]
\end{equation}
where $t_1>0$ is some instant where the respective solution $(u(x,t),v(x,t))$ satisfies
$$0<u(x,t_1)<u^*\mbox{ and } v^*<v(x,t_1)<a.$$
The existence of $t_1$ is ensured by Theorem \ref{T2} $(ii)$ (recall we assuming $L\leq\sqrt{d_1}\pi$).

Now, we consider a travelling wave $(\Psi,\Phi)$ such that
\begin{equation}\label{ITW}
0<\Psi(x,t_1)<u(x,t_1)\mbox{ and } v(x,t_1)<\Phi(x,t_1)<a
\end{equation}
for $x\in (0,L)$. We take the controls
\begin{equation}\label{C2TW}
(\overline{c_u}(x,t),\overline{c_v}(x,t))=(\Psi(x,t),\Phi(x,t))
\end{equation}
for $(x,t)\in\{0,L\}\times (t_1,\infty)$. We can use  Theorem \ref{cpao} to conclude that
$$\Psi(x,t)\leq u(x,t)\leq u^*\mbox{ and } v^*\leq v(x,t)\leq \Phi(x,t)$$
for $(x,t)\in (0,L)\times (t_1,\infty)$. It follows that $(u,v)$ goes  to $(u^*,v^*)$ as $t\to\infty$ uniformly in $[0,L]$, and the controllability is obtained with the controls defined in \eqref{CTW} and \eqref{C2TW}, namely
$$(c_u(x,t),c_v(x,t))=\left\{\begin{array}{ll}
   (\tilde{c_u}(x,t),\tilde{c_v}(x,t)) & (x,t)\in\{0,L\}\times (0,t_1]\\
     (\overline{c_u}(x,t),\overline{c_v}(x,t)) & (x,t)\in\{0,L\}\times (t_1,\infty).
\end{array}\right.$$

\subsection{Barrier solutions as a control strategy}

As stated in the Introduction, even the non-controllability results presented in theorems \ref{T4} and \ref{T3} can be significant for alternative control strategies. This is due to the underlying cause of the non-controllability, namely the existence and, more importantly, the identification of barrier solutions.

For example, for the single-specie $(1,0)$ in the Theorem \ref{T4}, the barrier  is $(u^s,v^s)$ (solution of \eqref{SPN}). This means that, if for a given objective, it is expected that the system does not evolve to this single-species state regardless of the quantities of species at the boundary of the region, then it is sufficient to start from an initial state \((u_0,v_0)\) such that $0\leq u_0(x)\leq u^s(x)$  and  $a\geq v_0(x)\geq v^s(x)$ for all $x\in [0,L]$.

The same applies to Theorem \ref{T3}. If, for instance, the goal is to prevent the extinction of both species, we can construct barrier functions for this purpose since
\begin{equation}\label{HE}L>\max\{\sqrt{d_1/(1-ak_1)}\pi,\sqrt{d_2/(a-k_2)}\pi\}.\end{equation}
 The proof of this theorem shows that \(\eta_1\) (solution of \eqref{ET2}) is a barrier for \(u\). Similarly, \(\eta_2\), the solution of 
 \begin{equation}\label{ET2S}
\left\{\begin{array}{l}
d_2\eta''(x)+\eta(x)(a- \eta(x)-k_2)=0\ \  x\in (0,L) \\
\eta(0)=\eta(L)=0,\\
\eta(x)>0,\ \  x\in (0,L),
\end{array}\right.
\end{equation}
whose existence is guaranteed by Lemma \ref{lemp}, is a barrier for \(v\). Simulations for this case are implemented in the Section \ref{NS}.

\section{Controllability   in finite time}\label{FT}

In this section our goal is to reach the target $(u^*,v^*)$ in finite time with controls satisfying the constraints \eqref{res}.

Here, the strategy is to start with the controls used in Theorem \ref{T1} and then, when the trajectory is sufficiently close to $(u^*,v^*)$, we use the control derived from a well-known result of exact local controllability, which we state below.

\begin{theorem}\label{T9}
Let $T>0$ fixed. Then there exist constants $C(T)$ and $\delta(T)$ such that, for all targets $(\overline{u},\overline{v})$ steady states of \eqref{SP} and for each initial data $(u_0,v_0)$ ($0\leq u_0\leq 1$, $0\leq v_0\leq a$) in $L^{\infty}(0,L)$ with
$$||(u_0,v_0)-(\overline{u},\overline{v})||_{L^{\infty}}\leq\delta$$
we can find a control $(c_u,c_v)\in L^{\infty}(\{0,L\}\times (0,T))\times L^{\infty}(\{0,L\}\times (0,T))$ such that the solution $(u,v)$ of \eqref{SP} starting at $(u_0,v_0)$ with control $(c_u,c_v)$ satisfies
$$(u(\cdot,T),v(\cdot,T))=(\overline{u},\overline{v}).$$
Furthermore,
$$\max\{|c_u(x,t)-\overline{u}(x)|,\  |c_v(x,t)-\overline{v}(x)|;\ (x,t)\in\{0,L\}\times (0,T)\}\leq C ||(u_0,v_0)-(\overline{u},\overline{v})||_{L^{\infty}}.$$
\end{theorem}

Theorem above is an adaptation of Lemma 2.1 from \cite{PGZ} to our case.

\begin{remark}
We note that the controls $(c_u,c_v)$ provided by Theorem \ref{T9} do not necessarily lie in $[0,1]\times [0,a]$.  As stated in the introduction, the values of $c_u$ and $c_v$ oscillate in a neighborhood of the desired target over time. Therefore, the only possible target that can be reached in finite time using this result is $(u^*,v^*)$, which lies between the bounds established by \eqref{res}.

\end{remark}

Now, we can state the main result of this section.

\begin{theorem}\label{T1F}
If \begin{equation}\label{CCF}k_2<a<\dfrac{1}{k_1}\end{equation}
then \eqref{SP} is controllable in finite time towards $(u^*,v^*)$ defined in \eqref{DH}.
\end{theorem}
\begin{proof}
    The proof is an application of theorems \ref{T1} and \ref{T9}. First, we fix $T>0$ and consider $C(T)$ and $\delta(T)$ given in Theorem \ref{T9} when applied in our problem \eqref{SP}. Given an initial condition $(u_0,v_0)$, we use Theorem \ref{T1} to ensure that there exists a control $(\overline{c}_u,\overline{c}_v)$ satisfying \eqref{res} such that the solution $(u,v)$ of \eqref{SP} starting from $(u_0,v_0)$ with control $(\overline{c}_u,\overline{c}_v)$ satisfies
$$||(u(\cdot,t_1),v(\cdot,t_1))-(u^*,v^*)||_{L^{\infty}}\leq \epsilon \leq\delta(T)$$
for some $t_1>0$, where $\epsilon>0$ is small enough such that 
\begin{equation}\label{epsi}
    0\leq u^*-C(T)\epsilon,\ \ 0\leq v^*-C(T)\epsilon, \ \ u^*+C(T)\epsilon\leq 1\mbox{ and } v^*+C(T)\epsilon\leq a.
\end{equation}
 Now, for $t>t_1$, we use Theorem \ref{T9} to obtain a control $(\tilde{c}_u(x,t),\tilde{c}_v(x,t))$ such that
$$(u(\cdot,t_1+T),v(\cdot,t_1+T))=(u^*,v^*).$$

We note that by Theorem \ref{T9} and \eqref{epsi},
$$0\leq u^*(x)-C(T)\epsilon\leq c_u(x,t)\leq u^*(x)+C(T)\epsilon\leq 1$$
and
$$0\leq v^*(x)-C(T)\epsilon\leq c_v(x,t)\leq u^*(x)+C(T)\epsilon\leq a$$
for all $(x,t)\in\{0,L\}\times (0,t_1+T)$. 

Denoting $t_1 + T$ by $\tilde{T}$, we define the control $(c_u,c_v)\in L^{\infty}(\{0,L\}\times (0,\tilde{T}))\times L^{\infty}(\{0,L\}\times (0,\tilde{T}))$ that drives the initial condition $(u_0,v_0)$ to the target $(u^*,v^*)$ in finite time $\tilde{T}$ by
\begin{equation}
(c_u(x,t),c_v(x,t))=\left\{\begin{array}{ll}
(\overline{c}_u(x,t),\overline{c}_v(x,t)),& (x,t)\in\{0,L\}\times (0,t_1)\\
(\tilde{c}_u(x,t),\tilde{c}_v(x,t)),& (x,t)\in\{0,L\}\times (t_1,\tilde{T}).
\end{array}\right.
\end{equation}
\end{proof}


\section{Extension to a higher dimension}\label{HD}
In this section, we discuss the same problem in a smooth bounded domain $\Omega\subset\mathbb{R}^n$ ($n>1$). In fact, theorems similar to those stated so far are also valid in the $n$-dimensional case. In summary, Theorem \ref{TN} (as well as Theorems \ref{cos}, \ref{cpao} and Lemma \ref{lemp}) is valid in $\Omega$ (see the references \cite{YY,pao,ZPAO,PAY}), and thus the proofs that use the behavior of solutions with Neumann boundary conditions can be repeated for the $n$-dimensional case.

Similarly, the role of the interval length $L$ appears as a consequence of the relationship between $L$ and the smallest eigenvalue of problem \eqref{AS} (see \eqref{EL} also), which can be considered in $\Omega$, namely
\begin{equation}\label{ASO}
\left\{\begin{array}{ll}
\Delta\phi(x)+\lambda\phi(x)=0,& x\in\Omega\\
\phi(x)=0,& x\in\partial\Omega.
\end{array}\right.
\end{equation}

In what follows, our main results are stated for the $n$-dimensional case. Theorem \ref{T1} can be stated in an identical manner, and if $\lambda_0$ is the smallest eigenvalue of \eqref{ASO}, then theorems \ref{T2}, \ref{T4}, \ref{T3}, and \ref{T5} can be written, respectively, as below.

\begin{theorem}\label{T2O}

\begin{enumerate}[$(i)$]
\item If $\lambda_0\geq a/d_2$ or
 $k_2>a$ then \eqref{SP} in $\Omega$ is controllable towards $(1,0)$.
 \item If $\lambda_0\geq 1/d_1$ or 
 $k_1>\dfrac{1}{a}$ then \eqref{SP} in $\Omega$ is controllable towards $(0,a)$.
 \end{enumerate}
\end{theorem}

\begin{theorem}\label{T4O}
If $k_2<a<1/k_1$ and
$$\lambda_0<\min\left\{\dfrac{1-ak_1}{d_1},\dfrac{a-k_2}{d_2}\right\}$$
then \eqref{SP} in $\Omega$ is not controllable towards either $(1,0)$ or $(0,a)$. 
\end{theorem}

\begin{theorem}\label{T3O}

  \begin{enumerate}[$(i)$]
  
  \item If  $$ \lambda_0\geq\max\left\{1/d_1,a/d_2\right\}$$
  then the system \eqref{SP} in $\Omega$ is controllable towards $(0,0)$;
  \item if 

  $$a<\dfrac{1}{k_1}\mbox{ and } \lambda_0<\dfrac{1-ak_1}{d_1}$$
  or
  $$a>k_2\mbox{ and }  \lambda_0<\dfrac{a-k_2}{d_2},$$
  then the system \eqref{SP} in $\Omega$ is not controllable towards $(0,0)$.
\end{enumerate}
\end{theorem}

\begin{theorem}\label{T5O}

  If $a=1$, $d_1=d_2=d>0$ and $ \lambda_0< 1/d$,
  then  \eqref{SP} in $\Omega$ is  controllable towards  an specific heterogeneous coexistence state $(u^{**},v^{**})$ defined in \eqref{uev} considered in $\Omega$.

\end{theorem}

Theorem \ref{T9} also holds in the domain $\Omega$, and thus the finite-time controllability result towards $(u^*,v^*)$ obtained in Theorem \ref{T1F} can be stated in the same way in $\Omega$.

The similarity of the results for the cases $n=1$ and $n>1$ led us to consider the proofs in the one-dimensional case, mainly for two reasons. The numerical simulations implemented in the following section, which illustrate our results, are better conducted and interpreted in the case $n=1$; moreover, in this case, the role of the domain size with respect to the controllability of the system becomes evident. However, it is well known that the smallest eigenvalue of \eqref{ASO} continuously depends on $\Omega$  (see \cite{SMO,HA}, for instance). In particular, when $\Omega$ is convex, we have that
\begin{equation}
\label{HA}
c(n)/\rho_{\Omega}^2\leq\lambda_0\leq C(n)/\rho_{\Omega}^2
\end{equation}
where $c(n)$, $C(n)$ are constants that depend only on the dimension $n$ and  $\rho_{\Omega}$ is the {\it inradius} of $\Omega$; that is, the radius of the largest ball contained in $\Omega$, 
$$\rho_{\Omega}:=\sup\{r>0; \exists x\in\Omega,\ \  B(x,r)\subset\Omega\}.$$

The inequality \eqref{HA} can be seen in \cite[Theorem 7.75]{HA}. Finally, we observe from the theorems above that controllability occurs in $\Omega$ if it has a small inradius, and does not occur when the inradius is large. Theorem \ref{T5O} is an exception to this interpretation, as it is a consequence of the global stability of the problem with zero Dirichlet boundary conditions. This analysis is similar to the one-dimensional case, and it corresponds to what is expected for controls acting only on the boundary of the domain. A similar conclusion was recently obtained for scalar reaction-diffusion equations of the bistable type \cite{Z2}.

\section{Numerical simulations}\label{NS}

In this section, we validate some of the results obtained above by simulating the behavior of the solutions as we alter the relations between the interval length \( L \) and the parameters of the problem. Our main objectives are:

\begin{itemize}
    \item to compare the control strategies obtained with a computationally derived optimal control strategy, and
    \item  to visualize the action of barrier solutions inhibiting control for specific targets.
\end{itemize} 

The simulations were performed using the {\it Mathematica} software and the {\it Casadi} package in {\it MATLAB}.

\subsection{Controllability towards a state of coexistence.}

In Theorem \ref{T1}, we saw that if there exists a coexistence state \((u^*,v^*)\), then it is controllable. As demonstrated, the strategy is to use the solution of the corresponding problem with Neumann boundary conditions as the control. In Figure \ref{FCON}, we have a simulation with the following parameters: \(d_1=d_2=0.01\), \(k_1=0.8\), \(k_2=0.7\), \(a=1\), \(L=1\) and $t\in\{1,5,12,18\}$. In this case, we have \eqref{CC} and 
$$(u^*,v^*)=\left(\dfrac{1-k_1a}{1-k_1k_2},\dfrac{a-k_2}{1-k_1k_2}\right)\approx (0.45,0.68).$$


Additionally, we also show the behavior of the control $(c_u,c_v)$ at \(x=0\) and \(x=1\) for $t\in (0,18)$.


\begin{figure}[ht!]
  \centering
   \includegraphics[width=0.99\linewidth]{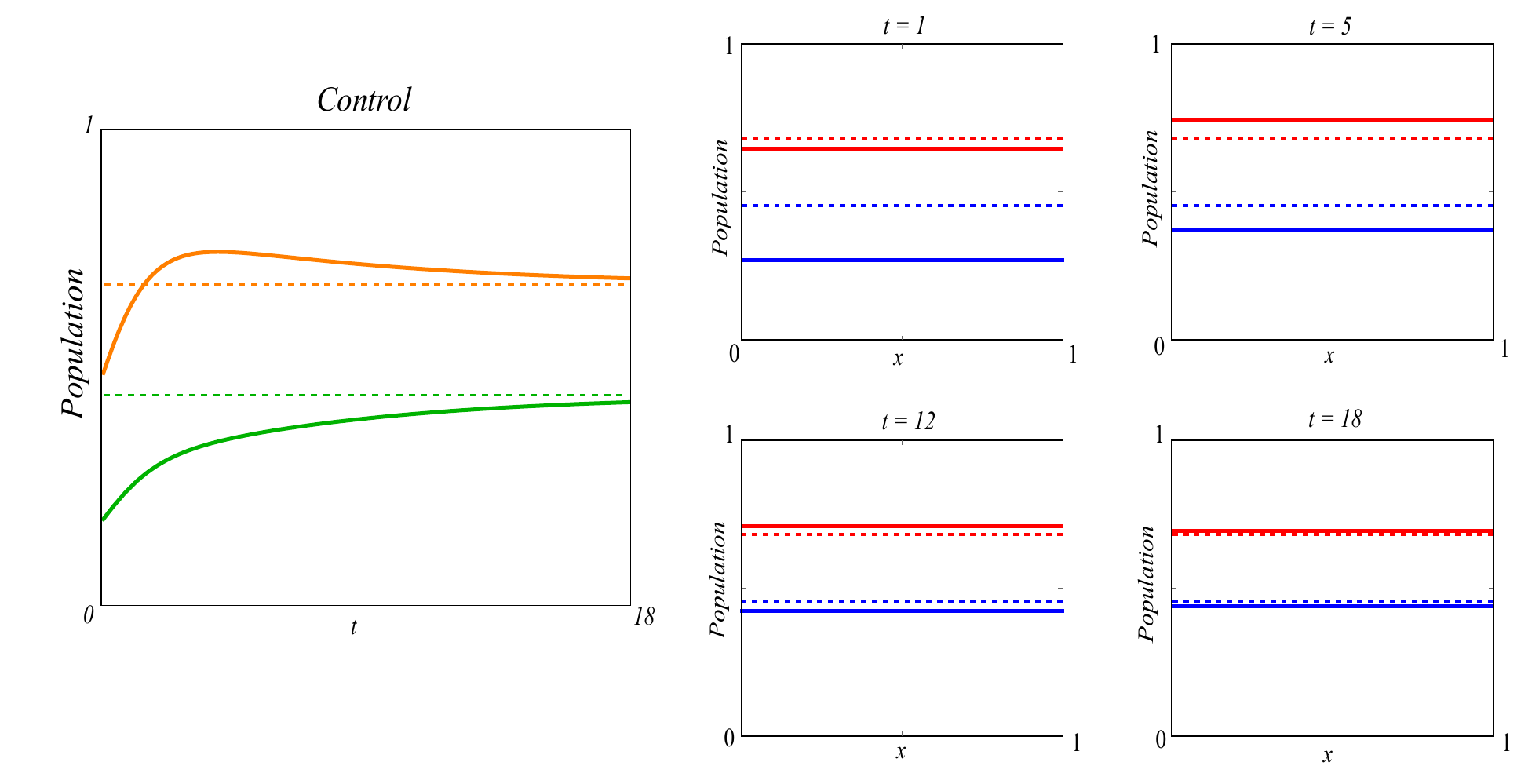}
 \caption{Controls \(c_u\) (green line) and \(c_v\) (orange line) approach \(u^*\) (dashed green line) and \(v^*\) (dashed orange line), respectively, for $x\in\{0,L\}$ (left). Solutions \(u\) (blue line) and \(v\) (red line) approach \(u^*\) (dashed blue line) and \(v^*\) (dashed red line), respectively  (right).
}
   \label{FCON}
\end{figure}

The initial condition assumed was \((u_0, v_0) = (0.2, 0.5)\), and we can observe that, for each fixed \(t\), the solutions arising from this strategy are constant functions of \(x\). This is not the case when we simulate the optimal control of the problem with the same parameters. Figure \ref{test2}  illustrates the behavior of the solutions and the controls.

\begin{figure}[ht!]
  \centering
   \includegraphics[width=0.99\linewidth]{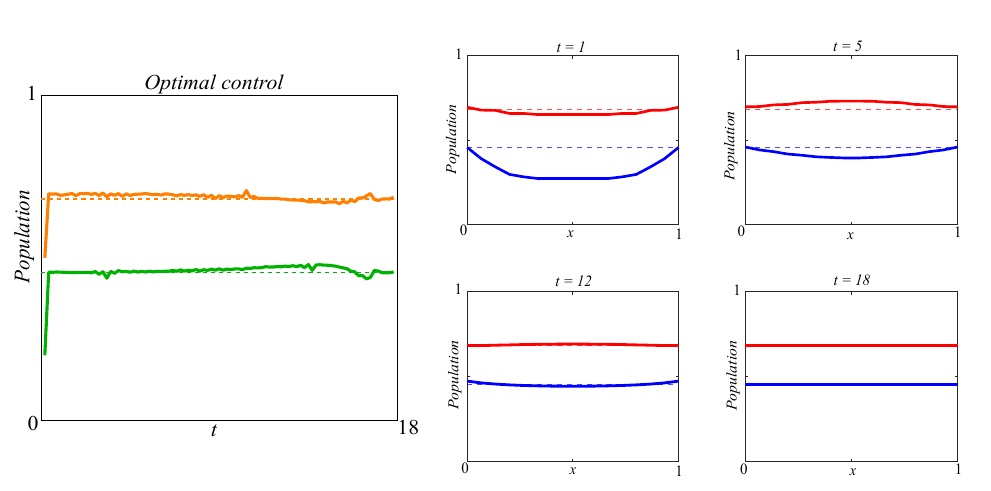}
 \caption{Optimal controls \(c_u\) (green line) and \(c_v\) (orange line) approach \(u^*\) (dashed green line) and \(v^*\) (dashed orange line), respectively, for $x\in\{0,L\}$ (left). Solutions \(u\) (blue line) and \(v\) (red line) approach \(u^*\) (dashed blue line) and \(v^*\) (dashed red line), respectively (right).}
   \label{test2}
\end{figure}



We observe that, although the controls behave similarly, and starting from the same initial state, the solutions approach the target \((u^*, v^*)\) in different ways.


\subsection{Controllability towards the survival of one of the species.}

The simulations in this case will be performed only for the target \((1,0)\). Everything could be analogously done for the target \((0,a)\).

Theorem \ref{T2} provides us with two distinct conditions for controllability towards $(1,0)$, each with a different strategy. In what follows, we simulate both strategies and also compare them with the optimal control obtained computationally.

The non-controllability results for $(1,0)$ and $(0,a)$ presented in Theorem \ref{T4} can also be simulated. The existence and visualization of the barrier solution allow us to interpret this phenomenon, which prevents the solutions from reaching these targets as long as the initial conditions are appropriately chosen. However, this analysis is similar to the one conducted for non-controllability towards $(0,0)$ (Subsection \ref{AEE}), so we will omit it here.

In the Figure \ref{test3}, we simulate the behavior of the solutions assuming the condition $a<k_2$, more precisely: $d_1=d_2=0.01$, $k_1=0.8$, $k_2=0.7$, $a=0.6$, $L=1$ for $t\in \{1,2,10,30\}$. The controls are derived from the associated Neumann problem, and their behaviors also can be seen in the Figure \ref{test3}.


\begin{figure}[ht!]
  \centering
   \includegraphics[width=0.8\linewidth]{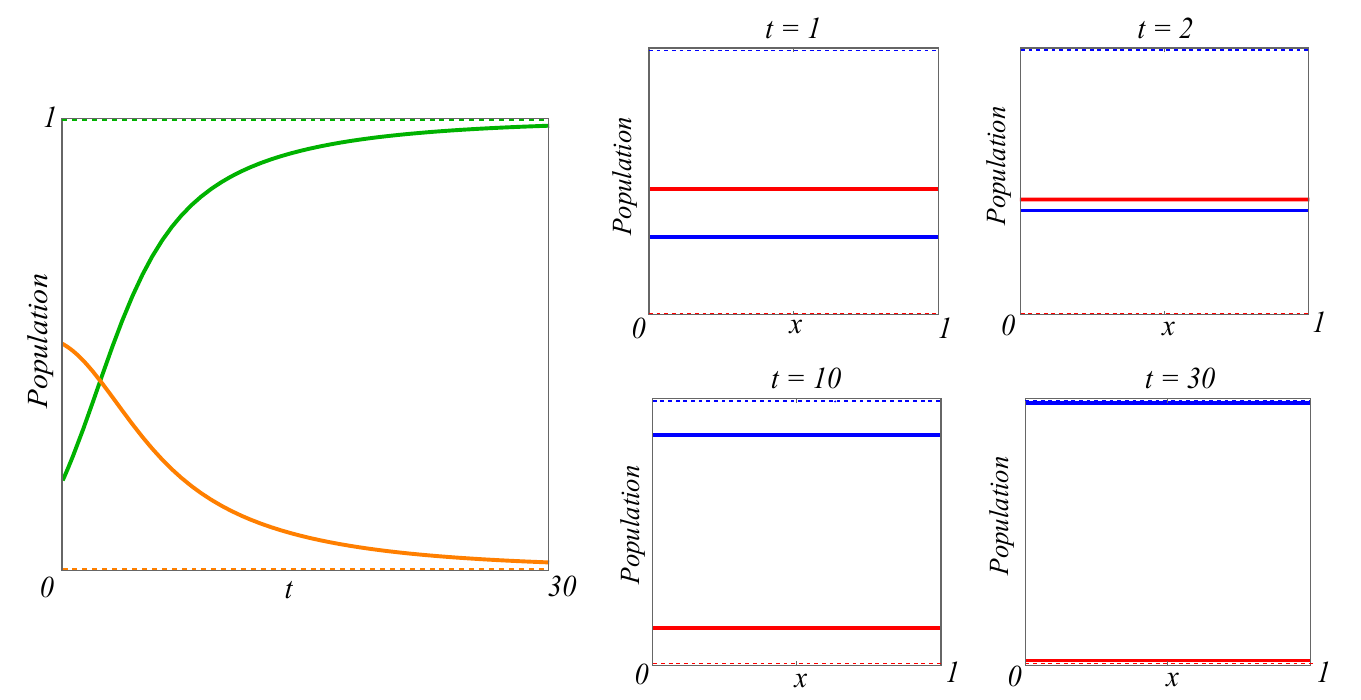}
 \caption{Controls \(c_u\) (green line) and \(c_v\) (orange line) approach $1$ and $0$, respectively, for $x\in\{0,L\}$ (left). Solutions \(u\) (blue line) and \(v\) (red line) approach $1$ (dashed blue line) and $0$ (dashed red line), respectively (right).
}
   \label{test3}
\end{figure}



In the Figure \ref{FUZCC}, we simulate the case $L<\sqrt{d_2/a}\pi$ ($d_1=0.01$, $d_2=3$ $k_1=0.8$, $k_2=0.7$, $a=0.6$, and $L=1$) with constant controls $c_u=1$, $c_v=0$ and $t\in\{0.005,0.05,2,4\}$. We observe that convergence to the target is faster compared to the previous case, especially for the species $v$. In both scenarios, we assume initial conditions $(u_0,v_0)=(0.2,0.5)$.

\begin{figure}[htbp!]
  \centering
   \includegraphics[width=0.99\linewidth]{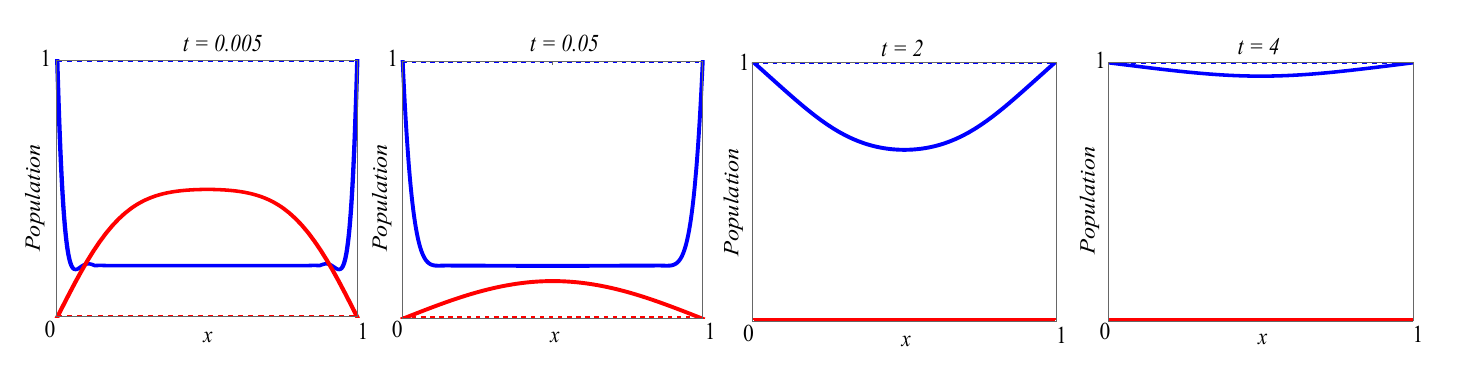}
 \caption{Solutions \(u\) (blue line) and \(v\) (red line) approach $1$ (dashed blue line) and $0$ (dashed red line), respectively (case $L<\sqrt{d_2/a}\pi$ and controls $c_u=1$ and $c_v=0$).
}
   \label{FUZCC}
\end{figure}



Figure \ref{test5} shows the behavior of the solutions under the optimal control obtained for the target \((1,0)\) for $t\in\{0.5,2,4,10\}$ and the behavior of this optimal control.


\begin{figure}[htbp!]
  \centering
   \includegraphics[width=0.9\linewidth]{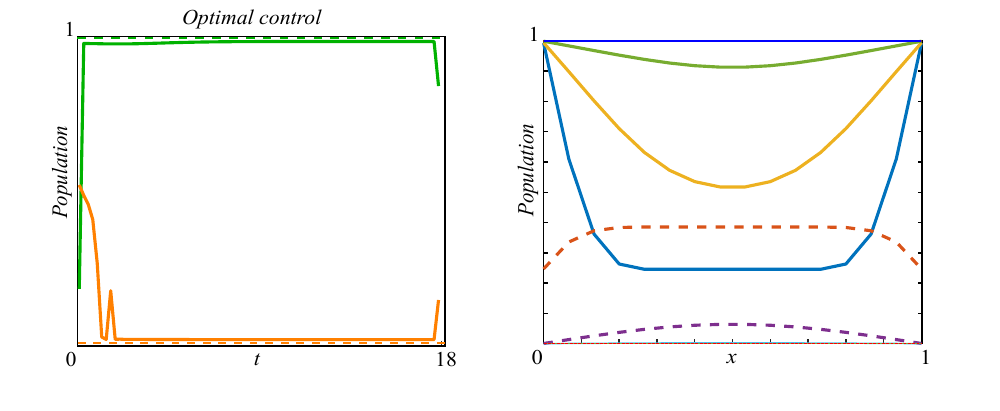}
 \caption{Solutions \( u \) (thin lines) and \( v \) (thin dashed lines) approach the target \((1,0)\) within the time intervals 
$t\in\{0.5,2,4,18\}$ (right), under the optimal controls  \(c_u\) (green line) and \(c_v\) (orange line) (left).
}
   \label{test5}
\end{figure}

The optimal control obtained above was achieved assuming \( L > \sqrt{d_2/a}\pi \). A similar behavior, not shown here, occurs when we assume \( a < k_2 \).


\subsection{Barrier solutions that prevent the extinction of one or both species.}\label{AEE}

In this subsection, we will simulate the case of species extinction. Simulations similar to the previous ones can be conducted for the controllability results we have in Theorem \ref{T3} $(i)$. However, we will dedicate this subsection to simulations related to non-controllability, specifically the existence of barrier solutions (Theorem \ref{T3} $(ii)$).

First, let us assume \(a<\dfrac{1}{k_1}\) and \(L>\sqrt{\dfrac{d_1}{1-ak_1}}\pi\), and then we will have the formation of a barrier solution that prevents \(u\) from approaching 0. Figure \ref{FBE} was generated with the parameters: \(k_1=0.8\), \(k_2=0.7\), \(a=1\), \(d_1=0.01\), \(d_2=4\), and \(L=1\).


\begin{figure}[htbp!]
  \centering
   \includegraphics[width=1\linewidth]{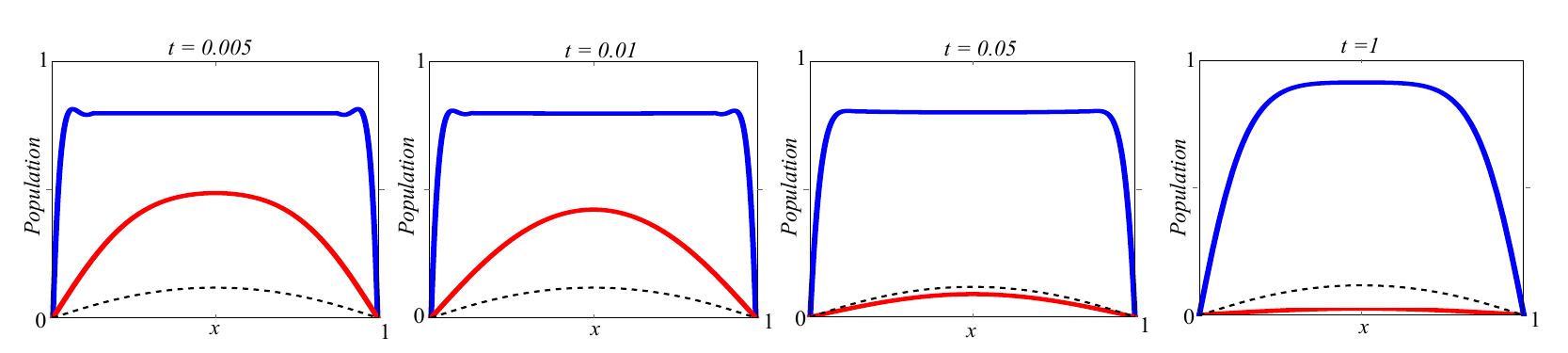}
 \caption{Solutions \(u\) (blue line), \(v\) (red line) and the barrier solution  (dashed black line) (case $\sqrt{d_1/(1-ak_1)}\pi<L<\sqrt{d_2/(a-k_2)}\pi$).
}
   \label{FBE}
\end{figure}

Note that, in this simulated case, the barrier solution only prevents the extinction of species \(u\) with the adopted initial condition, namely: \((u_0,v_0)=(0.8, 0.5)\) (\(u_0=0.8\) is greater than the barrier solution \(\eta(x)\), see \eqref{ET2}). This occurs because we simulated $\sqrt{d_1/(1-ak_1)}\pi<L<\sqrt{d_2/(a-k_2)}\pi$.

In this case, we used the control \((c_u,c_v)\equiv (0,0)\); however, the same phenomenon can be observed for any admissible control adopted. In other words, we can say that under these conditions, it is not possible to interfere with the populations on the boundary of the domain in such a way that leads species \(u\) to extinction.

Obviously, one can construct a barrier solution only for species \(v\), assuming  $$\sqrt{d_1/(1-ak_1)}\pi>L>\sqrt{d_2/(a-k_2)}\pi,$$ just as it is possible to construct barrier solutions simultaneously for \(u\) and \(v\) by assuming \eqref{HE}.
 In this case, with appropriate initial conditions taken, neither of the species will go extinct.

\begin{figure}[htbp!]
  \centering
   \includegraphics[width=0.99\linewidth]{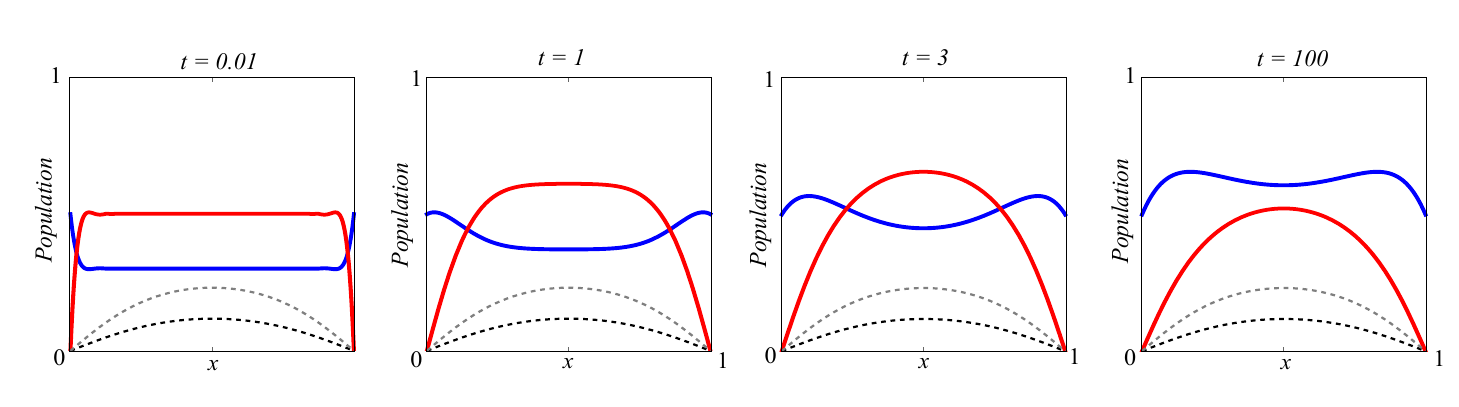}
 \caption{Solutions \(u\) (blue line), \(v\) (red line) and the barrier solutions to $u$ and $v$  (dashed black line and dashed gray line, respectively).
}
   \label{FBE2}
\end{figure}

In the above simulation, the adopted initial condition was \((u_0,v_0)=(0.3, 0.5)\). The barrier solutions represented by the dashed lines are \(\eta_1(x)\) (black), the solution of \eqref{ET2}, and \(\eta_2(x)\) (gray), the solution of \eqref{ET2S}.

\section{Concluding remarks and perspectives}\label{CR}

In this study, we have demonstrated how boundary control within an environment can be a significant factor in ecological management. We proved that, in the weak competition scenario considered, it is possible to steer solutions towards various proposed targets. The weak competition between species intuitively leads us to expect a viable situation of coexistence among them. In fact, the necessary condition for a homogeneous coexistence state  (see \eqref{DH}) is sufficient to control the system; this is what we proved in Theorem \ref{T1}.
We also found that it is possible to drive solutions towards states of a single species or even total extinction, provided the environment is sufficiently small or the parameters satisfy certain conditions (see Theorem \ref{T2} and Theorem \ref{T3} $(i)$).

Furthermore, it is important to highlight the controllability towards a heterogeneous coexistence  species $(u^{**},v^{**})$  (Theorem \ref{T5}).  Assuming the same reproduction rates and diffusion capacities for $u$ and $v$, the weak competition between the species in a sufficiently large domain, allows controllability to this target. As we have seen, this result is a consequence of the global stability of the system with zero Dirichlet boundary conditions.

Conversely, when controllability  is not achievable, we have shown explicitly the existence of barrier functions that prevent trajectories from approaching certain targets. These results can be found in Theorems \ref{T4} and \ref{T3} $(ii)$.

Controllability in finite time towards $(u^*,v^*)$ is always possible when this steady state exists. In fact, this is the content of Section \ref{FT}. This is achieved using constrains  controls  and is only possible because the target $(u^*,v^*)$ has its components satisfying $0<u^*<1$ and $0<v^*<a$.

 Finally, we highlight the extension of the results to higher-dimensional domains in Section \ref{HD}, as well as the various numerical simulations conducted in Section \ref{NS} that support the obtained results.

The issues addressed here naturally lead us to envision new possibilities. Small changes in problem \eqref{SP} significantly alter the model as well as its dynamics and results. Some of these possible changes, which could be studied in a similar manner to the one done in this work, are more detailed below.
\begin{itemize}
\item Strong competition between the species ($k_1, k_2 > 1$) or in relation to just one of the species (for example, $0 < k_1 < 1 < k_2$). Numerical simulations not shown here indicate similar controllability results for these cases. However, the proofs must necessarily differ from those presented here, as the condition \( k_1,k_2<1 \) is essential to our arguments.

\item We can consider only one species with diffusion capacity, assuming, for example, $d_1 > 0$ and $d_2 = 0$. This case could model the competition for soil nutrients between a fixed plant (without diffusion) and a fungus with spores that can move through the soil or air. 
Again, new arguments must be used in this case, as \( d_1,d_2 > 0 \) is essential for us.
\item Furthermore, other relationships between the species can be studied, such as predator-prey, by assuming, for instance, that \( k_1 > 0 \) and \( k_2 < 0 \).

\end{itemize}

Obviously, the model can be considered more general and/or more realistic. In this case, new techniques should be used, and we list some interesting possibilities below.

\begin{itemize}
\item It is certainly a great challenge to consider  the case with $3$ or more interacting species.
\item More realistic models present variable spatial diffusivity; in this case, we should consider a  non-constant diffusivity term,

\begin{equation}\label{SPCR}
\left\{\begin{array}{ll}
u_t=(\kappa_1(x) u_x(x,t))_x+u(1-u-k_1v),& (x,t)\in (0,L)\times\mathbb{R}^+\\
v_t=(\kappa_2(x) v_x(x,t))_x+v(a-v-k_2u),& (x,t)\in (0,L)\times\mathbb{R}^+\\
u(x,t)=c_u(x,t),& (x,t)\in\{0,L\}\times\mathbb{R}^+\\
v(x,t)=c_v(x,t),& (x,t)\in\{0,L\}\times\mathbb{R}^+\\
u(x,0)=u_0(x),\ \ v(x,0)=v_0(x),& x\in (0,L)
\end{array}\right.
\end{equation}
 with $\kappa_i(x) > 0$ for all $x$. The non-homogeneities $\kappa_i(x)$ interfere with the species diffusion capacity and bring new possibilities to the problem \cite{SHI}.
\item We can combine control strategies acting in the interior and on the boundary of the domain, for example
\begin{equation}\label{SP2CR}
\left\{\begin{array}{ll}
u_t=d_1u_{xx}+u(1-u-k_1v)+m_u(x,t),& (x,t)\in (0,L)\times\mathbb{R}^+\\
v_t=d_2v_{xx}+v(a-v-k_2u)+m_v(x,t),& (x,t)\in (0,L)\times\mathbb{R}^+\\
u(x,t)=c_u(x,t),& (x,t)\in\{0,L\}\times\mathbb{R}^+\\
v(x,t)=c_v(x,t),& (x,t)\in\{0,L\}\times\mathbb{R}^+\\
u(x,0)=u_0(x),\ \ v(x,0)=v_0(x),& x\in (0,L)
\end{array}\right.
\end{equation}
where $m_u$, $m_v$  are also controls that can represent species released into the environment.

\item Finally, we can also study other aspects of this problem. For example, the optimal control simulations conducted here (see figures \ref{test2} and \ref{test5}) indicate the {\it turnpike phenomenon}, in which the optimal control remains close to a stationary state for most of the time. There is no theoretical analysis regarding this observation for the problem considered here. More details on this topic can be found in \cite{TZ,oc1} and references therein.
\end{itemize}

\section{Acknowledgements}
M. Sonego has been partially supported by the Conselho Nacional de Desenvolvimento Científico
e Tecnológico (CNPq), Grant/Award Number: 311893/2022-8; Fundação de Amparo à
Pesquisa do Estado de Minas Gerais (FAPEMIG), Grant/Award Number: RED-00133-21 and
CAPES/Humboldt Research Scholarship Program, Number: 88881.876233/2023-01.
E. Zuazua has been funded by the Alexander von Humboldt-Professorship program, the ModConFlex Marie Curie Action, HORIZON-MSCA-2021-DN-01, the COST Action MAT-DYN-NET, the Transregio 154 Project Mathematical Modelling, Simulation and Optimization Using the Example of Gas Networks of the DFG, AFOSR  24IOE027 project, grants PID2020-112617GB-C22, TED2021-131390B-I00 of MINECO and PID2023-146872OB-I00 of MICIU (Spain),
Madrid Government - UAM Agreement for the Excellence of the University Research Staff in the context of the V PRICIT (Regional Programme of Research and Technological Innovation).


\end{document}